\documentclass[11pt]{amsart}
\usepackage[english]{babel}
\usepackage[latin1]{inputenc}
\usepackage{amstext}
\usepackage{amsmath}
\usepackage{amsfonts}
\usepackage{amssymb}
\usepackage{amsthm}
\usepackage[all]{xy}
\xyoption{arc}
\CompileMatrices
\usepackage{graphicx}

\newtheorem{thm}{Theorem}[section]
\newtheorem{cor}[thm]{Corollary}
\newtheorem{lem}[thm]{Lemma}
\newtheorem{prop}[thm]{Proposition}
\newtheorem{dfn}[thm]{Definition}
\newtheorem{rem}[thm]{Remark}

\newenvironment{pr}[1][Proof]{\noindent\textbf{#1.} }{\ \rule{0.5em}{0.5em}}

\newcommand{\Spec}{\operatorname{Spec}}

\newcommand{\Hom}{\operatorname{Hom}}

\newcommand{\Sym}{\operatorname{Sym}}

\newcommand{\im}{\operatorname{im}}

\newcommand{\rk}{\operatorname{rk}}

\newcommand{\GL}{\operatorname{GL}}
\newcommand{\SL}{\operatorname{SL}}
\newcommand{\PGL}{\operatorname{PGL}}

 \selectfont

\numberwithin{equation}{section}

\begin{document}

\title[Harder-Narasimhan filtration rk $2$ tensors and stable coverings] {Harder-Narasimhan filtration for rank $2$ tensors and stable coverings}
\author[A. Zamora]{Alfonso Zamora}

\address{Instituto de Ciencias Matem\'aticas (CSIC-UAM-UC3M-UCM),
Nicol\'as Cabrera 13-15, Campus Cantoblanco UAM, 28049 Madrid,
Spain}

\address{Departamento de Matem\'atica, Instituto Superior T\'ecnico, Av. Rovisco Pais
1049-001 Lisboa, Portugal}

\email{alfonsozamora@icmat.es, alfonsozamora@tecnico.ulisboa.pt}

\begin{abstract}
We construct a Harder-Narasimhan filtration for rank $2$ tensors,
where there does not exist any such notion a priori, as coming
from a GIT notion of maximal unstability. The filtration
associated to the 1-parameter subgroup of Kempf giving the maximal
way to destabilize, in the GIT sense, a point in the parameter
space of the construction of the moduli space of rank $2$ tensors
over a smooth projective complex variety, does not depend on
certain integer used in the construction of the moduli space, for
large values of the integer. Hence, this filtration is unique and
we define the Harder-Narasimhan filtration for rank $2$ tensors as
this unique filtration coming from GIT. Symmetric rank $2$ tensors
over smooth projective complex curves define curve coverings lying
on a ruled surface, hence we can translate the stability condition
to define stable coverings and characterize the Harder-Narasimhan
filtration in terms of intersection theory.

$ $

\textbf{Keywords:} Harder-Narasimhan filtration, GIT, tensors,
curve coverings, moduli space, Kempf

$ $

\textbf{Mathematics Subject Classification (2010):} 14D20, 14J60, 14L24, 14E20

\end{abstract}

\maketitle

\section*{Introduction}
This article is part of the research developed in the author's Ph.D.
Thesis (c.f. \cite{Za2}) and it is a continuation of \cite{GSZ}.
In a moduli problem, usually, we impose a notion of stability for
the objects in order to obtain a moduli space with good
properties. When constructing the moduli space using Geometric
Invariant Theory, a notion of GIT stability for the orbits appears
and, to obtain the moduli space it is shown, at some point, that
both notions of stability do coincide.

Harder and Narasimhan construct a canonical filtration (c.f.
\cite{HN}) for unstable sheaves, named after them, which maximally
contradicts the definition of stability we impose in the
construction of the moduli space. On the other hand, there has
been some results in the literature, trying to find the best way
of destabilizing an orbit in the GIT sense (c.f. \cite{GIT, He,
Ke}). The GIT stability is checked by $1$-parameter subgroups, by
the classical Hilbert-Mumford criterion, and it turns out that
there exists, up to some rescaling, a unique $1$-parameter
subgroup giving a notion of maximal GIT unstability. That special unique $1$-parameter subgroup produces a
filtration in a natural way, which we call Kempf filtration, based
on results of \cite{Ke} (c.f. section \ref{kempfsection}). The immediate question is whether the
Harder-Narasimhan filtration and the Kempf filtration do coincide.

In \cite{GSZ} the authors develop a idea to answer positively the previous question and establish a
correspondence between both filtrations, based on rewriting a function (which appears in \cite{Ke} and is
maximized by the special $1$-parameter subgroup) in a more geometrical way, to show that the Kempf filtration
satisfies certain convexity properties (c.f. section \ref{convexcones}), similar to the properties which characterize the Harder-Narasimhan
filtration. In this article, the author translates that idea to the case of
rank $2$ tensors over a smooth complex projective variety of arbitrary dimension.

We call a rank $2$ tensor the pair given by
$(E,\varphi:\overbrace{E\otimes\cdots \otimes E}^{\text{s times}}\rightarrow M)$, where 
$E$ is a coherent torsion free sheaf of rank $2$ over a smooth projective complex variety $X$, and $M$ is a
line bundle over $X$. Apart from being a
geometrical object by itself, in the case when $X$ is a smooth
projective complex curve and $\varphi$ is symmetric they define
degree $s$ coverings of $X$ lying on the ruled surface
$\mathbb{P}(E)$, by considering the vanishing locus of the
morphism (c.f. section \ref{coverings}). Hence, a notion of
stability for rank $2$ tensors defines a notion of stable or
unstable covering of an algebraic curve embedded into a ruled
surface.

Besides, the importance of this rank $2$ tensors case is that, for
other moduli problems previously considered in \cite{GSZ,Za1,Za2}
(e.g. sheaves, holomorphic pairs, Higgs sheaves, quiver
representations), there is already a notion of Harder-Narasimhan
filtration for these objects, constructed analogously to the case
of sheaves (c.f. \cite[Theorem 1.3.4]{HL2}). However, for rank $2$
tensors there is no notion a priori of Harder-Narasimhan
filtration because we do not know, in principle, how to define a
quotient tensor (c.f. Remark \ref{whynot}), which is needed to
construct the filtration recursively, after finding a maximally
destabilizing subobject.

Rank $2$ tensors are a particular case of tensors or decorated sheaves (c.f. \cite{GS} and \cite{GLSS}). After this work
were finished, it appeared a paper by A. Pustetto (c.f. \cite{Pu}) where it was proved the existence of a maximal destabilizing object
for tensors which are $\epsilon$-semistable and for those which are $k$-semistable of rank $\leq 3$. Both $\epsilon$ and $k$-stability notions
are much easier to check (they only need to checked on subsheaves instead of filtrations)
hence the calculation of the quantity (\ref{rightmu}) is more straightforward. Indeed, it is precisely the rank $2$ case where the stability
notion considered in this article coincides with $\epsilon$ and $k$-stability (c.f. \cite[Lemma 46]{Pu}), hence stability is checked by 
subobjects (c.f. Definition \ref{stabilitytensorsrk2}) and it is expected, as usual, the existence of a maximal destabilizing subobject.


The main technical difficulty here is to prove that the Kempf filtration does not depend on the choice of certain
integer made during the construction of the moduli space, for
large values of the integer (c.f. Theorem
\ref{kempfstabilizesrk2}, proof in section \ref{kstabilizes}).
Finally, we define the Harder-Narasimhan filtration as the unique
filtration (after proving the independence of this integer) giving
maximal unstability from the GIT point of view (c.f. section
\ref{HNdefinition}).

When the variety the tensor is defined over is a smooth projective
complex curve, we can characterize the Harder-Narasimhan
filtration in terms of intersection theory for ruled surfaces
(c.f. section \ref{coverings}). As we pointed out before, an
unstable tensor will define an unstable covering of the curve, and
the Harder-Narasimhan filtration can be reinterpreted as a section
of a ruled surface whose intersection numbers maximize certain
quantity. The Harder-Narasimhan filtration can be a useful tool to
study the moduli space of such coverings.

In principle, the ideas in this article can be used in different
moduli problems to show that the filtration giving maximal
unstability from the GIT point of view and the Harder-Narasimhan
filtration do coincide in cases where the latter is previously
known, or to define a notion of Harder-Narasimhan filtration,
otherwise. For example, in \cite{Za1}, a similar correspondence is
proven for representations of a finite quiver in the category of
finite dimensional vector spaces over an algebraically closed
field of arbitrary characteristic. For tensors in general, a
notion of Harder-Narasimhan filtration is unknown, and rank $2$
tensors is a particular example of tensors, for which this method
can be implemented.

{\bf Acknowledgments.} The author wishes to thank Usha Bhosle, Tom\'as L. G\'omez
 and Peter Newstead for useful discussions, and to the Isaac Newton Institute for
 Mathematical Sciences in Cambridge, United Kingdom, where
part of this work was done, for hospitality. This work has been
supported by project MTM2010-17389 and ICMAT Severo Ochoa project
SEV-2011-0087 granted by Spanish Ministerio de Econom\'ia y
Competitividad. The author was also supported by a FPU grant from
the Spanish Ministerio de Educaci\'on.

\renewcommand\thesubsection{\arabic{subsection}}

\section{Stability for rank $2$ tensors}
Let $X$ be a smooth complex projective variety of dimension $n$.
Let $E$ be a coherent torsion free sheaf over $X$, of rank $2$. Let $M$ be a line bundle over $X$. We
call a \emph{rank $2$ tensor} the pair consisting of
$$
(E,\varphi:\overbrace{E\otimes\cdots \otimes E}^{\text{s times}}\rightarrow M)\; ,
$$
where the morphism $\varphi$ is not identically zero. These objects are particular cases of the ones in \cite[Definition 1.1]{GS} for arbitrary $s$, $c=1$,
$b=0$, $R=\Spec \mathbb{C}$ and $\mathcal{D}=M$.

A weighted filtration $(E_\bullet,n_{\bullet})$ of a sheaf $E$ is a filtration
\begin{equation}
\label{filtE} 0 \subset E_1 \subset E_2 \subset\cdots \subset E_t \subset E_{t+1}=E,
\end{equation}
and rational numbers $n_{1},\, n_{2},\ldots , \,n_{t} > 0$. We
denote $r_{i}=\rk(E_{i})$. Let $\gamma$ be a vector of
$\mathbb{C}^r$ defined as $\gamma=\sum_{i=1}^{t}n_{i} \gamma^{(\rk
E_i)}$ where
$$\gamma^{(k)}:=\big( \overbrace{k-r,\ldots,k-r}^k,
 \overbrace{k,\ldots,k}^{r-k} \big)
\qquad (1\leq k < r) \, .$$ Hence, the vector is of the form
$$\gamma=(\overbrace{\gamma_{r_{1}},\ldots,\gamma_{r_{1}}}^{\rk E^1},
\overbrace{\gamma_{r_{2}},\ldots,\gamma_{r_{2}}}^{\rk E^2}, \ldots,
\overbrace{\gamma_{r_{t+1}},\ldots,\gamma_{r_{t+1}}}^{\rk E^{t+1}}) \; ,$$ where
$n_{i}=\frac{\gamma_{r_{i+1}}-\gamma_{r_{i}}}{r}$.

Now let $\mathcal{I}=\{1,...,t+1\}^{\times s}$ be the set of all multi-indexes $I=(i_{i},...,i_{s})$  and define
\begin{equation}
\label{rightmu} \mu(\varphi,E_{\bullet},n_{\bullet})=\min_{I\in
\mathcal{I}} \{\gamma_{r_{i_1}}+\cdots+\gamma_{r_{i_s}}:
\,\varphi|_{E_{i_1}\otimes\cdots\otimes E_{i_s}}\neq
0 \}.
\end{equation}
We assume that $\varphi$ is not identically zero, then
(\ref{rightmu}) is well defined. Let $I_{0}$ be the multi-index
giving minimum in (\ref{rightmu}). We will denote by
$\epsilon_i(\varphi,E_\bullet,n_{\bullet})$ (or just
$\epsilon_i(E_{\bullet})$ if the rest of the data is clear from
the context) the number of elements $k$ of the multi-index $I_{0}$
such that $r_{k}\leq r_{i}$. Let us call
$\epsilon^i(E_{\bullet})=\epsilon_{i+1}(E_{\bullet})-\epsilon_{i}(E_{\bullet})$.
Using a calculation made in \cite{GS, Za2}, we can rewrite
(\ref{rightmu}) as
\begin{equation}
\label{rightmu2} \mu(\varphi,E_\bullet,n_{\bullet})=\sum_{i=1}^{t}n_{i}(sr_{i}-\epsilon_{i}(E_{\bullet})r)\; .
\end{equation}

Let $\delta$ be a polynomial of degree at most $\dim X-1=n-1$ and positive leading coefficient. If $P_{1}$ and
$P_{2}$ are two polynomials, we write $P_{1}\prec P_{2}$ if $P_{1}(m)<P_{2}(m)$ for $m\gg 0$, and analogously
for "$\leq$" and "$\preceq$".

\begin{dfn}\cite[Definition 1.3]{GS}
\label{stabilityfortensors}
We say that
$(E,\varphi,u)$ is \emph{$\delta$-semistable} if for all weighted
filtrations $(E_{\bullet},n_{\bullet})$ of $E$,
\begin{equation}
\label{stabtensors} \sum_{i=1}^{t}n_{i}\big((rP_{E_{i}}-r_{i}P_{E})\big) + \delta
(sr_{i}-\epsilon_{i}(E_{\bullet})r)\big)\preceq 0\; .
\end{equation}
We say that $(E,\varphi,u)$ is \emph{$\delta$-stable} if we have a
strict inequality in (\ref{stabtensors}) for every weighted
filtration. If $(E,\varphi,u)$ is not $\delta$-semistable we say
that it is \emph{$\delta$-unstable}.
\end{dfn}

It suffices to check the condition in Definition \ref{stabilityfortensors} over filtrations with $\rk E_{i}<\rk
E_{i+1}$. Hence, as the rank of $E$ is $2$, the only filtrations we have to check are one-step filtrations, i.e.
subsheaves of rank $1$, and we can rewrite the stability condition as follows:

\begin{dfn}
\label{stabilitytensorsrk2}
A rank $2$ tensor $(E,\varphi)$ is \emph{$\delta$-semistable} if for
every rank $1$ subsheaf $L\subset E$
\begin{equation}
\label{stabilityrk2}
 (2P_{L}-P_{E})+\delta(s-2\epsilon(L))\leq 0,
\end{equation}
where $\epsilon(L)$ is the number of times that $1$ appears in the
multi-index $(i_1,\ldots,i_s)$ giving the minimum in
(\ref{rightmu}) (notice that $L$ plays the role of $E_{1}$ in (\ref{filtE})) and $P_{E}$, $P_{L}$ are the Hilbert polynomials
of $E$ and $L$ respectively. If the inequality is strict for every $L$, we say that $(E,\varphi)$ is \emph{$\delta$-stable}.
If $(E,\varphi)$ is not $\delta$-semistable, we say that it is \emph{$\delta$-unstable}.
\end{dfn}

\section{Moduli space of rank $2$ tensors}

We recall the main points of the construction of the moduli space
for tensors with fixed determinant $\det(E)\cong \Delta$ of degree
$d$ and $\rk E=2$. The construction for tensors in general
appears in \cite{GS}, following Simpson's method, and it is also
included in \cite[Section 1.2]{Za2}, following Gieseker's method. Recall that
our case can be obtained by setting $c=1$, $b=0$, arbitrary $s$,
$R=\Spec \mathbb{C}$ and $\mathcal{D}=M$, line bundle over $X\times R\simeq X$. 

Let $V$ be a vector space of dimension $p:=h^0(E(m))$, where $m$
is a suitable large integer (in particular, $E(m)$ is generated by
global sections and $h^i(E(m))=0$ for $i>0$). Given an isomorphism
$V\cong H^0(E(m))$ we obtain a point
$$(\overline{Q},\overline{\Phi}) \in
\mathbb{P}(\Hom(\wedge^r V , A) )\times \mathbb{P}(\Hom(V^{\otimes s},B))\; .
$$
If we change the isomorphism $\det(E)\cong \Delta$, we obtain a
different point in the line defined by $Q$. Similarly, if we
change the isomorphism $V\cong H^0(E(m))$ by a homothecy, we
obtain a different point in the line defined by $Q$. In both
cases, the point $\overline{Q}$ in the projective space is the
same. The same applies for $\overline{\Phi}$. If we fix once and
for all a basis of $V$, then giving an isomorphism between $V$ and
$H^0(E(m))$ is equivalent to giving a basis of $H^0(E(m))$. A
change of basis is given by an element of $\GL(V)$, but, since an
homothecy does not change the point
$(\overline{Q},\overline{\Phi})$, when we want to get rid of this
choice it is enough to divide by the action of $\SL(V)$.

A weighted filtration $(V_\bullet,n_\bullet)$ of $V$ is a filtration
\begin{equation}
\label{filtVrk2} 0 \subset V_1 \subset V_2 \subset \;\cdots\; \subset V_t \subset V_{t+1}=V,
\end{equation}
and rational numbers $n_1,\, n_2,\ldots , \,n_t > 0$. Similarly to
weighted filtrations of $E$ (c.f. (\ref{filtE})), this is
equivalent to giving a $1$-parameter subgroup $\Gamma: \mathbb{C}^*\to
\SL(V)$ represented by the vector
$$
\Gamma=(\overbrace{\Gamma_1,\ldots,\Gamma_1}^{\dim V^1},
\overbrace{\Gamma_2,\ldots,\Gamma_2}^{\dim V^2},
\ldots,
\overbrace{\Gamma_{t+1},\ldots,\Gamma_{t+1}}^{\dim V^{t+1}}) \; ,
$$ up to conjugacy by an element of the parabolic subgroup defined by the filtration, where
$n_{i}=\frac{\Gamma_{i+1}-\Gamma_{i}}{\dim V}$ and define $V^{i}:=V_{i}/V_{i-1}$ (c.f. \cite{GSZ,Za2}).

By the Hilbert-Mumford criterion (c.f. \cite[Theorem 2.1]{GIT}), a point
$$(\overline{Q},\overline{\Phi}) \in
\mathbb{P}(\Hom(\wedge^r V , A) )\times \mathbb{P}(\Hom(V^{\otimes s},B))
$$
is \emph{GIT semistable} with respect to the natural linearization
on $\mathcal{O}(a_1,a_2)$ if and only if, for all weighted filtrations, it
is
$$
\mu(\overline{Q},V_\bullet,n_\bullet) + \frac{a_2}{a_1} \mu(\overline{\Phi},V_\bullet, n_\bullet) \leq 0\; .$$
The second summand of the expression is given by
\begin{equation}
\label{rightmuV} \mu(\Phi,V_{\bullet},n_{\bullet})=\min_{I\in \mathcal{I}} \{\Gamma_{\dim
V_{i_1}}+\cdots+\Gamma_{\dim V_{i_s}}: \,\Phi|_{V_{i_1}\otimes\cdots\otimes V_{i_s}}\neq 0 \}.
\end{equation}
If $I=(i_1,\ldots,i_s)$ is the multi-index giving minimum in (\ref{rightmuV}), we will analogously denote by
$\epsilon_i(\overline{\Phi},V_\bullet,n_\bullet)$ (or just $\epsilon_i(\overline{\Phi})$ if the rest of the data
is clear from the context) the number of elements $k$ of the multi-index $I$ such that $\dim V_k\leq \dim V_i$.
Let $\epsilon^i(\overline\Phi)=\epsilon_i(\overline\Phi)-\epsilon_{i-1}(\overline\Phi)$.

Given a weighted filtration of $V\simeq H^{0}(E(m))$ as in (\ref{filtVrk2}), we denote by 
$E_{V_i}$ the subsheaf of $E$ generated by $V_i$, and let $r_{i}=\rk E_{V_{i}}$ be its rank . Similarly, we denote by $E_{V^{i}}$ 
the sheaf generated by $V^{i}\simeq V_{i}/V_{i-1}$ and let $r^{i}=\rk E_{V^{i}}=r_{i}-r_{i-1}$ be its rank.

Using the calculation of the numerical function to apply Mumford criterion for GIT stability (c.f.
\cite[Proposition 1.2.29]{Za2}), we can state the following:
\begin{prop}
\label{GITstabrk2}
A point $(\overline{Q},\overline{\Phi})$ is \emph{GIT $a_2/a_1$-semistable} if for all
weighted filtrations $(V_\bullet,n_\bullet)$,
$$
\sum_{i=1}^{t}  n_i \big( r \dim V_i - r_i \dim V)
+\frac{a_2}{a_1} ( s\dim V_i-\epsilon_i(\overline\Phi)\dim V
)\big) \leq 0\; .
$$
\end{prop}

The following theorem can be found in \cite[Theorem 4.5.3]{GLSS}. See also \cite[Theorem 3.6.]{GS} and \cite[Theorem 1.2.31]{Za2}
for the construction of a moduli space of tensors following Simpson (c.f. \cite{Si}) and Gieseker (c.f. \cite{Gi}) 
respectively, whose notations we follow in this article.

\begin{thm}
\label{GIT-deltark2} Let $(E,\varphi)$ be
a tensor. There exists an $m_{0}$ such that, for $m\geq m_{0}$ the
associated point $(\overline{Q},\overline{\Phi})$ is GIT
$a_2/a_1$-semistable if and only if the tensor is
$\delta$-semistable, where
$$
\frac{a_2}{a_1} = \frac{r\delta(m)}{P_E(m)-s\delta(m)}\; .
$$
\end{thm}

\section{Kempf theorem}
\label{kempfsection}

Let $X$ be a smooth complex projective variety of dimension $n$
and let $\delta$ be a polynomial of degree at most $\dim X-1=n-1$
and positive leading coefficient. Let $(E,\varphi)$ be a
$\delta$-unstable rank $2$ tensor. Let $m_{0}$ be an integer as in
Theorem \ref{GIT-deltark2} (i.e. such that the $\delta$-stability
and the GIT stability coincide) and also such that $E$ is $m_{0}$
regular (choosing a larger integer, if necessary). Choose an
integer $m\geq m_{0}$ and let $V$ be a vector space of dimension
$P_{E}(m)=h^{0}(E(m))$.

By the Hilbert-Mumford criterion, stability of an orbit in the
parameter space where a group acts can be checked through
$1$-parameter subgroups, which turns out to be the checking of the
positivity of some quantity (c.f. Proposition \ref{GITstabrk2}).
The natural question which arises is whether there exists a best
way to destabilize a point in the sense of GIT, i.e. whether there
exists a best $1$-parameter subgroup which maximizes that
quantity. There are results in the literature (c.f. \cite{GIT},
\cite{He}, \cite{Ke}) studying the possibility of finding the best
$1$-parameter subgroup moving most rapidly toward the origin, i.e.
giving a notion of GIT maximal unstability. We will make use of
\cite{Ke} for our purposes.

Given a $1$-parameter subgroup, or equivalently a weighted
filtration, i.e. a filtration of vector subspaces $0\subset
V_{1}\subset \cdots \subset V_{t+1}= V$ and rational numbers
 $n_{1},\cdots,n_{t}>0$, we define the following function
$$\mu(V_{\bullet},n_{\bullet})=\frac{\sum_{i=1}^{t}  n_{i} \big(r\dim V_{i}-r_{i}\dim V+
\frac{a_{2}}{a_{1}}(s\dim V_{i}-\epsilon_{i}(\overline\Phi)\dim V
)\big)} {\sqrt{\sum_{i=1}^{t+1} {\dim V^{i}_{}} \Gamma_{i}^{2}}}\;
,$$ which we call a \emph{Kempf function} for this problem, i.e. a
function whose numerator coincides with the
numerical function in Proposition \ref{GITstabrk2} and the
denominator is a norm, a bilinear symmetric invariant form $||\Gamma ||$ in the space of
$1$-parameter subgroups (c.f. definition of length in \cite{Ke}). As the group $SL(V)$ is simple, every
bilinear symmetric invariant form is a multiple of the Killing
norm, hence it is unique up to scalar. Note that the norm is chosen in order to avoid the rescaling of the weights when
asking for the maximum of the function. In other words, the choice
of a norm calibrates the speed of the $1$-parameter subgroups.

The result of Kempf states that, given a GIT unstable point, i.e.
a point for which there exists any $1$-parameter subgroup making
the quantity in Proposition \ref{GITstabrk2} positive (the
numerator of the Kempf function), there exists a unique parabolic
subgroup containing a unique $1$-parameter subgroup in each
maximal torus, giving maximum for the Kempf function. In terms of
filtrations, there exists a unique weighted filtration giving
maximum for the Kempf function. Therefore, we rewrite
\cite[Theorem 2.2]{Ke} for this case:

\begin{thm}
\label{kempftheoremsheaves} There exists a unique weighted filtration
$$0\subset V_{1}\subset \cdots \subset V_{t+1}= V$$
and rational numbers $n_{1},\cdots,n_{t}>0$, up to multiplication by a scalar, called the \emph{Kempf filtration
of V}, such that the Kempf function $\mu(V_{\bullet},n_{\bullet})$ achieves the maximum among all filtrations
and positive weights $n_{i}>0$.
\end{thm}

Let
\begin{equation}
\label{vkempf-filtrk2}
0\subset V_{1}\subset \cdots \subset V_{t+1}= V
\end{equation} be the \emph{Kempf filtration} of $V$
(c.f. Theorem \ref{kempftheoremsheaves}), and let
\begin{equation}
\label{ekempf-filtrk2}
0\subseteq (E^{m}_{1},\varphi|_{E_{1}^{m}})\subseteq
(E^{m}_{2},\varphi|_{E_{2}^{m}})\subseteq\cdots
(E^{m}_{t},\varphi|_{E_{t}^{m}})\subseteq
(E^{m}_{t+1},\varphi|_{E_{t+1}^{m}})=(E,\varphi)
\end{equation}
be the \emph{$m$-Kempf filtration} of the rank $2$ tensor $(E,\varphi)$, where $E_{i}^{m}\subset E$ is the
subsheaf generated by $V_{i}$ under the evaluation map. Note that the subsheaves do depend on the integer $m$ we
have chosen during the process of constructing the moduli space.

For a given $m$, the $m$-Kempf filtration represents the maximal
way of destabilizing a $\delta$-unstable tensor from the GIT point
of view. In this case, there is no notion, a priori, of a
Harder-Narasimhan filtration. Hence, the filtration we obtain from
GIT, once we prove that it does not depend on $m$, will define by
uniqueness a notion of Harder-Narasimhan filtration.

In the following we will prove this Theorem, in an analogous way as it was done in \cite{GSZ} for sheaves.
\begin{thm}
\label{kempfstabilizesrk2} There exists an integer $m'\gg 0$ such that the $m$-Kempf filtration of the $\rk 2$
tensor $(E,\varphi)$ does not depend on $m$, for $m\geq m'$.
\end{thm}

\section{Results on convexity}
\label{convexcones}

Now we recall the results from \cite[Section 2]{GSZ} about
convexity. We study a function on a convex set, and how to
maximize it. It will turn out to be that this function will be in
correspondence with the Kempf function and we will use these
results to figure out properties about the Kempf filtration.

Endow $\mathbb{R}^{t+1}$ with an inner product
$(\cdot,\cdot)$ defined by a diagonal matrix
 $$
 \left(
 \begin{array}{ccc}
 b^1 & & 0 \\
  & \ddots & \\
 0 & & b^{t+1}\\
 \end{array}
 \right)
 $$
where $b^i$ are positive integers. Let
$$
\mathcal{C}= \big\{ x\in \mathbb{R}^{t+1} : x_1<x_2<\cdots <x_{t+1}
\big\}\; ,
$$
and let $v= (v_1,\cdots,v_{t+1})\in \mathbb{R}^{t+1}-\{0\}$
verifying $\sum_{i=1}^{t+1} v_{i}b^{i}=0$. Define the function
\begin{eqnarray*}
  \label{eq:mu}
\mu_{v}:\overline{\mathcal{C}}-\{0\} & \to & \mathbb{R}\\
\Gamma & \mapsto & \mu_{v}(\Gamma)=\frac{(\Gamma,v)}{||\Gamma||}
\end{eqnarray*}
and note that  $\mu_{v}(\Gamma)=||v||\cdot \cos(\Gamma,v)$. Then, the function $\mu_{v}(\Gamma)$ does not depend
on the norm of $\Gamma$ and takes the same value on every point of the ray spanned by each $\Gamma$.

Assuming that there exists $\Gamma\in\overline{\mathcal{C}}$ verifying $\mu_v(\Gamma)>0$ we would like to find a
vector $\Gamma\in \overline{\mathcal{C}}$ maximizing $\mu_{v}$. We set $w^i=-b^iv_i$, $w_i=w^1+\cdots+w^i$, and
$b_i=b^1+\cdots+b^i$ and draw a graph joining the points with coordinates $(b_i,w_i)$, where the slope of each
segment is given by $-v_i$ (thin line in Figure \ref{convexfigure}). Now draw the convex envelope of this graph
(thick line in Figure \ref{convexfigure}), whose coordinates are denoted by $(b_i,\widetilde{w_i})$, and let us
define $\Gamma_{i}=-\frac{\widetilde{w_{i}}-\widetilde{w_{i-1}}}{b^i}$. Hence, $-\Gamma_i$ are the slopes of the
convex envelope graph of $v$, and call the vector defined in this way $\Gamma_{v}$.

\setlength{\unitlength}{1cm}

\begin{figure}[h]
   \begin{center}
\begin{picture}(13,7)(-1,-1)
\thicklines
\put(0,0){\line(1,0){11.5}}
\put(0,0){\line(0,1){5.5}}
\put(6,-0.3){\makebox(0,0)[c]{$b_i$}} \put(-0.4,3){\makebox(0,0)[c]{\rotatebox{90}{$w_i,\widetilde{w_i}$}}}
\put(0,0){\makebox(0,0){$\circ$}} \put(0,0){\line(1,1){4}} \put(2,2){\makebox(0,0){$\circ$}}
\put(3,3){\makebox(0,0){$\circ$}} \put(4,4){\makebox(0,0){$\circ$}} \put(4,4){\line(4,1){4}}
\put(5,4.25){\makebox(0,0){$\circ$}} \put(8,5){\makebox(0,0){$\circ$}} \put(8,5){\line(1,-1){2}}
\put(9.5,3.5){\makebox(0,0){$\circ$}} \put(10,3){\makebox(0,0){$\circ$}} \put(10,3){\line(1,-3){1}}
\put(11,0){\makebox(0,0){$\circ$}} \thinlines \put(0,0){\makebox(0,0){$\circ$}} \put(0,0){\line(2,1){2}}
\put(2,1){\makebox(0,0){$\circ$}} \put(2,1){\line(1,1){1}} \put(3,2){\makebox(0,0){$\circ$}}
\put(3,2){\line(1,2){1}} \put(4,4){\line(1,-1){1}} \put(5,3){\makebox(0,0){$\circ$}} \put(5,3){\line(3,2){3}}
\put(9,4){\makebox(0,0){$\circ$}} \put(9,4){\line(1,-4){0.5}} \put(9.5,2){\makebox(0,0){$\circ$}}
\put(9.5,2){\line(1,2){0.5}} \put(2,0.6){\makebox(0,0)[l]{$(b_1,w_1)$}}
\put(1.8,2.2){\makebox(0,0)[r]{$(b_1,\widetilde{w_1})$}} \put(3,1.6){\makebox(0,0)[l]{$(b_2,w_2)$}}
\put(2.8,3.2){\makebox(0,0)[r]{$(b_2,\widetilde{w_2})$}}
\put(3.8,4.2){\makebox(0,0)[r]{$(b_3,\widetilde{w_3}=w_3)$}}
\put(5,4.7){\makebox(0,0)[c]{$(b_4,\widetilde{w_4})$}} \put(5,2.6){\makebox(0,0)[c]{$(b_4,w_4)$}}
\put(8,5.4){\makebox(0,0)[c]{$(b_5,\widetilde{w_5}=w_5)$}}
\end{picture}
\end{center}
\caption{Convex envelope $\Gamma_{v}$ of $v$} \label{convexfigure}
\end{figure}
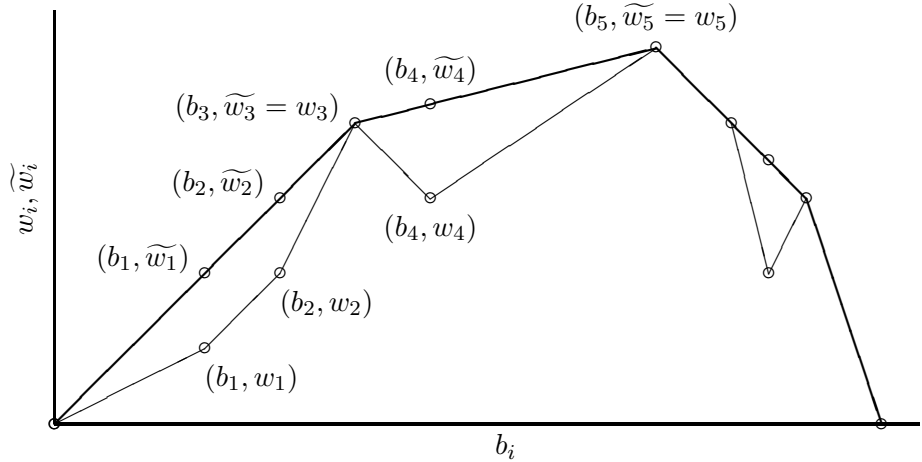

\begin{thm}
\label{maxconvexenvelope} The vector $\Gamma_{v}$ defined in this way (c.f. Figure \ref{convexfigure}) gives a
maximum for the function $\mu_v$ on its domain.
\end{thm}

\section{The $m$-Kempf filtration stabilizes with $m$}
\label{kstabilizes} In this section we will prove Theorem
\ref{kempfstabilizesrk2} through a series of partial results.
Given a $\delta$-unstable rank $2$ tensor $(E,\varphi)$ we have
the $m$-Kempf filtration of $(E,\varphi)$ (c.f.
(\ref{ekempf-filtrk2})). To this filtration we associate a graph,
in order to apply the previous results on convexity.

\begin{dfn}
\label{graphrk2} Let $m\geq m_{0}$. Given $0\subset V_{1}\subset
\cdots \subset V_{t+1}= V$, a filtration of vector spaces of $V$,
let
$$v_{m,i}=m^{n+1}\cdot \frac{1}{\dim V^{i}\dim V}\big[r^{i}\dim V-r\dim V^{i}+\dfrac{a_{2}}{a_{1}}(\epsilon^{i}(\overline\Phi)\dim V-s\dim V^{i})\big]\; ,$$
$$b_{m}^{i}=\dfrac{1}{m^{n}}\dim V^{i}>0\; ,$$
$$w_{m}^{i}=-b_{m}^{i}\cdot v_{m,i}=m\cdot \frac{1}{\dim V}\big[r\dim V^{i}-r^{i}\dim V+\dfrac{a_{2}}{a_{1}}(s\dim V^{i}-\epsilon^{i}(\overline\Phi)\dim V) \big]\; .$$
Also let
$$b_{m,i}=b_{m}^{1}+\ldots +b_{m}^{i}=\dfrac{1}{m^{n}}\dim V_{i}\; ,$$
$$w_{m,i}=w_{m}^{1}+\ldots +w_{m}^{i}=m\cdot \frac{1}{\dim V}\big[r\dim V_{i}-r_{i}\dim V+\dfrac{a_{2}}{a_{1}}(s\dim V_{i}-\epsilon_{i}(\overline\Phi)\dim V)\big]\; .$$
We call the graph defined by points $(b_{m,i},w_{m,i})$ the \emph{graph associated to the filtration}
$V_{\bullet}\subset V$.
\end{dfn}

Now we prove a crucial Lemma which will let us relate the Kempf
function with the function in Theorem \ref{maxconvexenvelope}, in
order to prove Theorem \ref{kempfstabilizesrk2}. The lemma
strongly uses the assumption on the rank $2$ of the tensor, the
reason why the result cannot be analogously extended in more
generality. A discussion about the issues when applying the method
for rank $3$ can be read in \cite[Section 2.5]{Za2}.

\begin{lem}
\label{independenceofweightsrk2} The symbols
$\epsilon_i(\overline\Phi)=\epsilon_i(\overline\Phi,V_{\bullet},n_{\bullet})$ do not depend on the weights
$n_{\bullet}$. Therefore, the graph associated to the filtration only depends on the data $V_{\bullet}\subset
V$, not the weights $n_{\bullet}$.
\end{lem}
\begin{proof}
Note that $\rk E_{1}\geq 1$ because it is generated by, at least, a
non zero global section. Suppose that $\rk E^{m}_{1}=\rk
E^{m}_{2}=\ldots=\rk E^{m}_{k}=1$ and $\rk E^{m}_{k+1}=\ldots=\rk
E^{m}_{t}=\rk E=2$. Then, for example, $E^{m}_{1}$ coincide with
$E^{m}_{2}$ on an open set and, generically, the behavior with
respect to $\varphi$ is the same, i.e.
$$\overline\Phi|_{V_{1}\otimes\cdots\otimes
V_{1}}=0\Leftrightarrow \varphi|_{E^{m}_{1}\otimes\cdots\otimes
E^{m}_{1}}=0\Leftrightarrow \varphi|_{E^{m}_{2}\otimes
E^{m}_{1}\cdots \otimes E^{m}_{1}}=0\; .$$ Therefore, the values
$\epsilon_i(\overline\Phi,V_{\bullet},n_{\bullet})$ only depend on
the filters $E^{m}_{i}$ but not on the specific values of the
$\Gamma_{i}$. In fact, they will only depend on $\Gamma_{1}$ and
$\Gamma_{k+1}$, because they are the minimal ones among the
filters of the same rank (c.f. (\ref{rightmu}) and
(\ref{rightmuV})). In this case we will just write
$\epsilon_i(\overline\Phi,V_{\bullet})$, or
$\epsilon_i(\overline\Phi)$, when the filtration is clear from the
context.
\end{proof}

Next, we shall identify the Kempf function in Theorem \ref{kempftheoremsheaves}
$$\mu(V_{\bullet},n_{\bullet})=\frac{\sum_{i=1}^{t}  n_{i} \big(r\dim V_{i}-r_{i}\dim V+\frac{a_{2}}{a_{1}}
(s\dim V_{i}-\epsilon_{i}(\overline\Phi)\dim V )\big)}{\sqrt{\sum_{i=1}^{t+1} {\dim V^{i}} \Gamma_{i}^{2}}}$$
$$=\frac {\sum_{i=1}^{t+1} \frac{\Gamma_i}{\dim V} \big( r^i \dim V - r\dim V^i+\frac{a_2}{a_1}
 (\epsilon^i(\overline\Phi)\dim V -s\dim V^i)\big)}
 {\sqrt{\sum_{i=1}^{t+1} {\dim V^{i}} \Gamma_{i}^{2}}}\; ,$$
where $n_{i}=\frac{\Gamma_{i+1}-\Gamma_{i}}{\dim V}$, with the function in Theorem \ref{maxconvexenvelope}.
Precisely, we use Lemma \ref{independenceofweightsrk2} to assure that the data of the filters
$V_{\bullet}\subset V$, and the data of the weights $n_{\bullet}$ are independent, so we can maximize the Kempf
function with respect to each of them, independently, as in Theorem \ref{maxconvexenvelope}.

\begin{prop}
\label{identificationrk2} For every integer $m$, the following equality holds
$$\mu(V_{\bullet},n_{\bullet})=m^{(-\frac{n}{2}-1)}\cdot \mu_{v_{m}}(\Gamma)$$
between the Kempf function on Theorem \ref{kempftheoremsheaves} and the function in Theorem
\ref{maxconvexenvelope}.
\end{prop}
\begin{pr}
By Lemma \ref{independenceofweightsrk2}, we can fix a vector $v_{m}$ and look for the maximum of the function
$\mu_{v_{m}}$ among the corresponding convex cone.
\end{pr}

In the following, we will omit the subindex $m$ for the numbers $v_{m,i}$, $b_{m,i}$, $w_{m,i}$ in the
definition of the graph associated to the filtration of vector spaces, where it is clear from the context.

Now, we recall (c.f. \cite{GSZ,Za2}) two lemmas encoding the convexity properties of the graph associated to the
Kempf filtration. They will be used in the following, to show properties shared by the possible filters
$E_{i}^{m}$ appearing in the different $m$-Kempf filtrations.

\begin{lem}\cite[Lemma 3.4]{GSZ} or \cite[Lemma 2.1.15]{Za2}
\label{lemmaA}
Let $0\subset V_{1}\subset \cdots \subset V_{t+1}= V$ be the Kempf filtration of $V$ (c.f.. Theorem
\ref{kempftheoremsheaves}). Let $v=(v_{1},...,v_{t+1})$ be the vector of the graph associated to this filtration
by Definition \ref{graphrk2}. Then
$$
v_{1}<v_{2}<\ldots<v_{t}<v_{t+1}\; ,
$$
i.e., \emph{the graph is convex}.
\end{lem}

\begin{lem}\cite[Lemma 3.5]{GSZ} or \cite[Lemma 2.1.16]{Za2}
\label{lemmaB} Let $0\subset V_{1}\subset \cdots \subset V_{t+1}=
V$ be the Kempf filtration of $V$ (c.f.. Theorem
\ref{kempftheoremsheaves}). Let $W$ be a vector space with
$V_{i}\subset W\subset V_{i+1}$ and consider the new filtration
$V'_{\bullet}\subset V$
\begin{equation}
    \begin{array}{ccccccccccccccccc}
    0 & \subset & V'_{1} & \subset & \cdots & \subset & V'_{i} & \subset & V'_{i+1} & \subset &
    V'_{i+2} & \subset & \cdots & \subset & V'_{t+2} & = & V\\
    || & & || & & & & || & & || & & || & & & & & & ||  \\
    0 & \subset & V_{1}& & & \subset & V_{i} & \subset & W & \subset & V_{i+1} & \subset &
    \cdots & \subset & V_{t+1} & = & V
    \end{array}
\end{equation}
Then, $v'_{i+1}\geq v_{i+1}$. We say that \emph{the Kempf filtration is the convex envelope of every
refinement}.
\end{lem}

\begin{lem} \cite[Corollary 1.7]{Si} or \cite[Lemma 2.2]{HL1}
\label{LePSi} Let $r>0$ be an integer. Then there exists a
constant $B$ with the following property: for every torsion free
sheaf $E$ with $0<\rk (E)\leq r$, we have
$$h^{0}(E)\leq \frac{1}{g^{n-1}n!}\big ((\rk(E)-1)([\mu_{max}(E)+B]_{+})^{n}+([\mu_{mim}(E)+B]_{+})^{n}\big)\; ,$$
where $g=\deg \mathcal{O}_{X}(1)$, $[x]_{+}=\max\{0,x\}$, and
$\mu_{max}(E)$ (respectively $\mu_{min}(E)$) is the maximum (resp.
minimum) slope of the Mumford-semistable factors of the
Harder-Narasimhan filtration of $E$.
\end{lem}

We denote
$$P_{\mathcal{O}_{X}}(m)=\frac{\alpha_{n}}{n!}m^{n}+\frac{\alpha_{n-1}}{(n-1)!}m^{n-1}+...
+\frac{\alpha_{1}}{1!}m+\frac{\alpha_{0}}{0!}$$ the Hilbert
polynomial of $\mathcal{O}_{X}$, then $\alpha_{n}=g$. Let
$$P(m)=\frac{rg}{n!}m^{n}+\frac{d+r\alpha_{n-1}}{(n-1)!}m^{n-1}+...$$
be the Hilbert polynomial of the sheaf $E$, where $d$ is the degree and $r$ is the rank. Let us call
$A=d+r\alpha_{n-1}$, so
$$P(m)=\frac{rg}{n!}m^{n}+\frac{A}{(n-1)!}m^{n-1}+...$$

Let us define
\begin{equation}
\label{C_constant_rk2} C=\max\{r|\mu _{\max }(E)|+\frac{d}{r}+r|B|+|A|+s\delta_{n-1}(n-1)!+1\;,\;1\},
\end{equation}
a positive constant, where $\delta_{n-1}$ is the leading coefficient of the polynomial $\delta(m)$, of degree $\leq n-1$ (if $\deg \delta< n-1$, set $\delta_{n-1}=0$).

\begin{prop}
\label{boundednessrk2} Given a sufficiently large $m$, each filter in the $m$-Kempf filtration of the $\rk 2$ tensor $(E,\varphi)$ has slope
$\mu(E^{m}_{i})\geq \dfrac{d}{r}-C$.
\end{prop}

\begin{pr}
The proof follows analogously to \cite[Proposition 4.8]{GSZ}. Choose an $m_{1}$ such that for $m\geq m_{1}$
$$[\mu_{max}(E)+gm+B]_{+}=\mu_{max}(E)+gm+B$$
and
$$[\frac{d}{r}-C+gm+B]_{+}=\frac{d}{r}-C+gm+B\; .$$
Let $m_{2}$ be such that $P_{E}(m)-s\delta(m)>0$ for $m\geq
m_{2}$. Now consider $m\geq \max\{m_{0},m_{1},m_{2}\}$ and let
$$0\subseteq (E^{m}_{1},\varphi|_{E_{1}^{m}})\subseteq
(E^{m}_{2},\varphi|_{E_{2}^{m}})\subseteq\cdots
(E^{m}_{t},\varphi|_{E_{t}^{m}})\subseteq
(E^{m}_{t+1},\varphi|_{E_{t+1}^{m}})=(E,\varphi)$$
be the $m$-Kempf filtration.

Suppose that we have a filter $E^{m}_{i}\subseteq E$, of rank $r_{i}$ and degree $d_{i}$, such that $\mu
(E^{m}_{i})<\frac{d}{r}-C$. The subsheaf $E_{i}^{m}(m)\subset E(m)$ satisfies the estimate in Lemma \ref{LePSi},
$$h^{0}(E^{m}_{i}(m))\leq \frac{1}{g^{n-1}n!}\big ((r_{i}-1)
([\mu_{max}(E^{m}_{i})+gm+B]_{+})^{n}+([\mu_{min}(E^{m}_{i})+gm+B]_{+})^{n}\big)\; .$$

Given that $\mu_{max}(E_{i}^{m})\leq \mu_{max}(E)$ and $\mu_{min}(E_{i}^{m})\leq \mu(E_{i}^{m})<\frac{d}{r}-C$,
and using the choice of $m$,
$$h^{0}(E^{m}_{i}(m))\leq \frac{1}{g^{n-1}n!}\big ((r_{i}-1)
(\mu_{max}(E)+gm+B)^{n}+(\frac{d}{r}-C+gm+B)^{n}\big)=G(m)\; ,$$
where
$$G(m)=\frac{1}{g^{n-1}n!}\big
[r_{i}g^{n}m^{n}+ng^{n-1}\big((r_{i}-1)\mu_{max}(E)+\frac{d}{r}-C+r_{i}B\big)m^{n-1}+\cdots
\big ]\; .$$

By Definition \ref{graphrk2}, to the $m$-Kempf filtration we
associate the graph given by
$$w_{j}=w^{1}+\ldots +w^{j}=m\cdot \frac{1}{\dim V}\big[r\dim V_{j}-r_{j}\dim V+\dfrac{a_{2}}{a_{1}}(s\dim V_{j}-\epsilon_{j}
(\overline\Phi)\dim V )\big]\; .$$ We will get a contradiction by showing that $w_{i}<0$. Indeed, if $w_{i}<0$
there is a $j<i$ such that $-v_{j}<0$. Hence, as the graph is convex by Lemma \ref{lemmaA} the rest of the
slopes of the graph are negative, $-v_{k}<0$, $k\geq i$. Then $w_{i}>w_{i+1}>\ldots w_{t+1}$, and $w_{t+1}<0$.
But it is
$$w_{t+1}=m\cdot \frac{1}{\dim V}\big[r\dim V_{t+1}-r_{t+1}\dim V+\dfrac{a_{2}}{a_{1}}(s\dim V_{t+1}-\epsilon_{t+1}
(\overline\Phi)\dim V )\big]=0\; ,$$
because $r_{t+1}=r$, $V_{t+1}=V$ and $\epsilon_{t+1}
(\overline\Phi)=s$, then the contradiction.

Since $E^{m}_{i}(m)$ is generated by $V_{i}$ under the evaluation map, it is $\dim V_{i}\leq
H^{0}(E^{m}_{i}(m))$, hence
$$w_{i}=\frac{m}{\dim V}\big[r\dim V_{i}-r_{i}\dim V+\dfrac{a_{2}}{a_{1}}(s\dim V_{i}-\epsilon_{i}(\overline\Phi)\dim V )\big]\leq$$
$$\frac{m}{P_{E}(m)}\big[rh^{0}(E^{m}_{i}(m))-r_{i}P_{E}(m)+\dfrac{r\delta(m)}{P_{E}(m)-s\delta(m)}(sh^{0}(E^{m}_{i}(m))-
\epsilon_{i}(\overline\Phi)P_{E}(m))\big]\leq$$
$$m\cdot \frac{\big[(P_{E}(m)-s\delta(m))(rG(m)-r_{i}P_{E}(m))+
(r\delta(m))(sG(m)-\epsilon_{i}(\overline\Phi)P_{E}(m))\big]}{P_{E}(m)(P_{E}(m)-s\delta(m))}\;
.$$

Then, $w_{i}<0$ is equivalent to
$$\Psi(m)=(P_{E}(m)-s\delta(m))(rG(m)-r_{i}P_{E}(m))+(r\delta(m))(sG(m)-\epsilon_{i}(\overline\Phi)P_{E}(m))<0\; ,$$
and $\Psi(m)=\xi_{2n}m^{2n}+\xi_{2n-1}m^{2n-1}+\cdots
+\xi_{1}m+\xi_{0}$ is a $(2n)^{th}$-order polynomial, whose higher order coefficient is
$$\xi_{2n}=(P_{E}(m)-s\delta(m))_{n}(rG(m)-r_{i}P_{E}(m))_{n}+(r\delta(m))_{n}(sG(m)-
\epsilon_{i}(\overline\Phi)P_{E}(m))_{n}=$$
$$(P_{E}(m)-s\delta(m))_{n}(r\frac{r_{i}g}{n!}-r_{i}\frac{rg}{n!})+0=0\; .$$
The $(2n-1)^{th}$-order coefficient is
$$\xi_{2n-1}=(P_{E}(m)-s\delta(m))_{n}(rG(m)-r_{i}P_{E}(m))_{n-1}+(r\delta(m))_{n-1}(sG(m)-
\epsilon_{i}(\overline\Phi)P_{E}(m))_{n}=$$
$$\frac{rg}{n!}(rG_{n-1}-r_{i}\frac{A}{(n-1)!})+r\delta_{n-1}(s\frac{r_{i}g}{n!}-
\epsilon_{i}(\overline\Phi)\frac{rg}{n!})\; ,$$ where $G_{n-1}$ is
the $(n-1)^{th}$-coefficient of the polynomial $G(m)$,
$$G_{n-1}=\frac{1}{g^{n-1}n!}ng^{n-1}((r_{i}-1)\mu_{max}(E)+\frac{d}{r}-C+r_{i}B)=$$
$$\frac{1}{(n-1)!}((r_{i}-1)\mu_{max}(E)+\frac{d}{r}-C+r_{i}B)\leq$$
$$\frac{1}{(n-1)!}((r_{i}-1)|\mu_{max}(E)|+\frac{d}{r}-C+r_{i}|B|)\leq$$
$$\frac{1}{(n-1)!}(r|\mu_{max}(E)|+\frac{d}{r}-C+r|B|)<\frac{-|A|}{(n-1)!}-s\delta_{n-1}\; ,$$
where last inequality comes from the definition of $C$ in
(\ref{C_constant_rk2}). Then
$$\xi_{2n-1}<\frac{rg}{n!}\big(r(\frac{-|A|}{(n-1)!}-s\delta_{n-1})-r_{i}\frac{A}{(n-1)!}\big )+r\delta_{n-1}\big(\frac{r_{i}gs}{n!}-
\epsilon_{i}(\overline\Phi)\frac{rg}{n!}\big)=$$
$$\frac{rg}{n!}\big [ \big (\frac{-r|A|-r_{i}A}{(n-1)!}\big )
-rs\delta_{n-1}+\delta_{n-1}(r_{i}s-\epsilon_{i}(\overline\Phi)r)\big
]=$$
$$\frac{rg}{n!}\big [ \big (\frac{-r|A|-r_{i}A}{(n-1)!}\big )
+\delta_{n-1}(-rs +r_{i}s-\epsilon_{i}(\overline\Phi)r)\big ]<\frac{rg}{n!}\delta_{n-1}(-rs
+r_{i}s-\epsilon_{i}(\overline\Phi)r)\; ,$$ because $-r|A|-r_{i}A<0$. Last expression is always negative because, 
if $r_{i}<r$, 
$$-rs+r_{i}s-\epsilon_{i}(\overline\Phi)r=-(r-r_{i})s-\epsilon_{i}(\overline\Phi)r\leq -\epsilon_{i}(\overline\Phi)r\leq 0\; , $$
with equality if and only if $r_{i}=r$, and if $r_{i}=r$, then $\epsilon_{i}(\overline\Phi)=s$, and 
$$-rs+r_{i}s-\epsilon_{i}(\overline\Phi)r=-rs<0\; .$$
Hence, it is $\xi_{2n-1}<0$.

Therefore $\Psi(m)=\xi_{2n-1}m^{2n-1}+\cdots +\xi_{1}m+\xi_{0}$
with $\xi_{2n-1}<0$, so there exists an integer $m_{3}$ such that for $m\geq
\{m_{0},m_{1},m_{2},m_{3}\}$ we have $\Psi(m)<0$ and
$w_{i}<0$, then the contradiction.
\end{pr}

Once we have seen that all possible filters in the different $m$-Kempf filtrations have their numerical
invariants bounded, and all of them are subsheaves of the same sheaf, we can prove the following:

\begin{prop}
\label{regularrk2} There exists an integer $m_{4}$ such that for
$m\geq m_{4}$ the sheaves $E^{m}_{i}$ and
$E^{m,i}=E^m_i/E^m_{i-1}$ are $m_{4}$-regular. In particular their
higher cohomology groups, after twisting with
$\mathcal{O}_{X}(m_{4})$, vanish and they are generated by global
sections.
\end{prop}

\begin{pr}
C.f. \cite[Proposition 3.9]{GSZ}.
\end{pr}

\begin{prop}
\label{taskrk2}
Let $m\geq m_{4}$. For each filter $E_{i}^{m}$ in the $m$-Kempf filtration of the $\rk 2$ tensor $(E,\varphi)$, we
have $\dim V_{i}=h^{0}(E_{i}^{m}(m))$, therefore $V_{i}\cong H^{0}(E_{i}^{m}(m))$.
\end{prop}

\begin{pr}
Let $V_{\bullet}\subseteq V$ be the Kempf filtration of $V$ (c.f..
Theorem \ref{kempftheoremsheaves}) and let
$(E_{\bullet}^{m},\varphi|_{E_{\bullet}^{m}})\subseteq
(E,\varphi)$ be the $m$-Kempf filtration of $(E,\varphi)$. We can
construct two filtrations:

\begin{equation}
 \label{filtrationVrk2}
    \begin{array}{ccccccccccccc}
    0 & \subset & \cdots & \subset & V_{i} & \subset & V_{i+1} & \subset & V_{i+2} & \subset & \cdots & \subset & V \\
    & & & & \cap & & || & & ||& & & &   \\
    & & & & H^{0}(E_{i}^{m}(m)) & \subset & H^{0}(E_{i+1}^{m}(m)) & \subset & H^{0}(E_{i+2}^{m}(m)) &  &  & &
    \end{array}
\end{equation}

and

\begin{equation}
\label{filtrationV'rk2}
    \begin{array}{ccccccccccccc}
    0 & \subset & \cdots & \subset & V_{i} & \subset & H^{0}(E_{i}^{m}(m)) & \subset & V_{i+1} & \subset & \cdots & \subset &  V\\
    & & & & || & & || & & || & & & &   \\
    & & & & V'_{i} & &  V'_{i+1} & & V'_{i+2} & & & &
    \end{array}
\end{equation}
to be in situation of Lemma
\ref{lemmaB}, where $W=H^{0}(E_{i}^{m}(m))$, filtration $V_{\bullet}$ is \eqref{filtrationVrk2} and filtration
$V'_{\bullet}$ is \eqref{filtrationV'rk2}.

Now, the graph associated to the filtration $V_{\bullet}$ is
given, by Definition \ref{graphrk2}, by the points
$$(b_{i},w_{i})=(\dfrac{\dim V_{i}}{m^{n}},\frac{m}{\dim V}\big(r\dim V_{i}-r_{i}\dim V+\dfrac{a_{2}}{a_{1}}(s\dim V_{i}-\epsilon_{i}
(\overline\Phi,V_{\bullet})\dim V))\big)\; ,$$
the slopes $-v_{i}$ of the graph given by

$$-v_{i}=\frac{w^{i}}{b^{i}}=\frac{w_{i}-w_{i-1}}{b_{i}-b_{i-1}}=
\frac{m^{n+1}}{\dim V}\big(r-r^{i}\frac{\dim V}{\dim V^{i}}+\frac{a_{2}}{a_{1}}(s-\epsilon^{i}(\overline\Phi,V_{\bullet})\frac{\dim V}
{\dim V^{i}})\big)\leq$$
$$\frac{m^{n+1}}{\dim V}\big(r+s\frac{a_{2}}{a_{1}}\big):=R\; ,$$
and equality holds if and only if $r^{i}=0$ (note that $r^{i}=0$ implies $\epsilon^{i}(\overline\Phi,V_{\bullet})=0$).

The new point which appears in the graph of the filtration
$V'_{\bullet}$ is
$$Q=\big(\dfrac{h^{0}(E_{i}^{m}(m))}{m^{n}},\frac{m}{\dim V}(rh^{0}(E_{i}^{m}(m))-r_{i}\dim V+\dfrac{a_{2}}{a_{1}}
(sh^{0}(E_{i}^{m}(m))- \epsilon_{i}(\overline\Phi,V_{\bullet})\dim
V))\big)\; .$$ Note that
\begin{equation}
\label{epsilon}
    \begin{array}{c}
\epsilon_{j}(\overline{\Phi},V'_{\bullet})=\epsilon_{j}(\overline{\Phi},V_{\bullet})\;,\;j\leq i\\
\epsilon_{j}(\overline{\Phi},V'_{\bullet})=\epsilon_{j-1}(\overline\Phi,V_{\bullet})\;,\;j>i.
    \end{array}
\end{equation}
This is the reason why we write $\epsilon_{i}(\overline\Phi,V_{\bullet})$ instead of
$\epsilon_{i}(\overline\Phi,V'_{\bullet})$.

The slope of the segment between $(b_{i},w_{i})$ and $Q$ is, similarly,
$$-v'_{i}=\dfrac{m^{n+1}}{\dim V}(r+s\dfrac{a_{2}}{a_{1}})=R\; .$$

By Lemma \ref{lemmaA}, the graph is convex, so $v_{1}<v_{2}<\ldots<v_{t+1}$. Besides, $r^{1}=r_{1}>0$, then
$-R<v_{1}$. This is because $E$ is torsion free, hence $E_{1}^{m}\subset E$ also has no torsion, and a rank $0$
torsion free sheaf is the zero sheaf. On the other hand, the graph associated to $V_{\bullet}'\subset V$ is a
refinement of the one associated to the Kempf filtration, $V_{\bullet}\subset V$, then we apply Lemma
\ref{lemmaB} and get $v'_{i}\geq v_{i}$. Hence,
$$-R<v_{1}<v_{2}<\ldots<v_{i}\leq v'_{i}=-R\; ,$$
which is a contradiction.

Therefore, $\dim V_{i}=h^{0}(E_{i}^{m}(m))$, for every filter in the $m$-Kempf filtration.
\end{pr}

\begin{cor}
\label{rankrk2} Let $m\geq m_{4}$. For every filter $E_{i}^{m}$ in
the $m$-Kempf filtration of the $\rk 2$ tensor $(E,\varphi)$, it
is $r^{i}>0$. Therefore, the $m$-Kempf filtration consists on a
rank $1$ subsheaf, $0\subset (L^{m},\varphi|_{L^{m}})\subset (E,\varphi)$.
\end{cor}
\begin{pr}
By Proposition \ref{taskrk2}, $r^{i}=0$ is equivalent to
$-v_{i}=R$. Then, $r^{1}=r_{1}>0$ and
$-R<v_{1}<v_{2}<\ldots<v_{t+1}$ imply the statement.
\end{pr}

For any $m\geq m_{4}$, by Corollary \ref{rankrk2} there is only one filter $(L^{m},\varphi|_{L^{m}})$ in the $m$-Kempf filtration and,
by Proposition \ref{regularrk2}, $L^{m}$ is $m_{4}$-regular. Hence, $L^{m}(m_{4})$ is generated by the subspace
$H^{0}(L^{m}(m_{4}))\subset H^{0}(E(m_{4}))$ by the evaluation map. Note that the
dimension of the vector space $H^{0}(E(m_{4}))$ does not depend on $m$.

We call \emph{$m$-type} of the $m$-Kempf filtration to the Hilbert
polynomial $P_{L^{m}}$. Once we fix $V\simeq H^{0}(E(m_{4}))$ whose
dimension does not depend on $m$, all possible filtrations of $V$
are parametrized by a finite-type scheme, hence the set of
possible $m$-types
$$\mathcal{P}=\big\{P_{L^{m}}\big\}$$
is finite, for all integers $m\geq m_{4}$.

Rewrite the graph associated to the $m$-Kempf filtration (c.f. Definition \ref{graphrk2})
$$v_{m,i}=\frac{m^{n+1}}{\dim V^{i}\dim V}\big[r^{i}\dim V-r\dim V^{i}+\dfrac{a_{2}}{a_{1}}(\epsilon^{i}
(\overline\Phi)\dim V-s\dim V^{i})\big]\; ,$$
$$b_{m}^{i}=\frac{1}{m^{n}}\cdot \dim V^{i}\; ,$$
as
$$v_{m,i}=\frac{m^{n+1}}{P_{m}^{i}(m)P(m)}\big[r^{i}P(m)-rP_{m}^{i}(m)+\dfrac{r\delta(m)}{P(m)-s\delta(m)}(\epsilon^{i}
(\overline\Phi)P(m)-sP_{m}^{i}(m))\big]\; ,$$
$$b_{m}^{i}=\frac{1}{m^{n}}\cdot P_{m}^{i}(m)\; ,$$
by Propositions \ref{regularrk2} and \ref{taskrk2}.

Note that, by Corollary \ref{rankrk2}, the graph has only two slopes given by
$$v_{m,1}=\frac{m^{n+1}}{P_{L^{m}}(m)P(m)}\big[P(m)-2P_{L^{m}}(m)+\dfrac{2\delta(m)}{P(m)-s\delta(m)}(\epsilon_{L^{m}}
P(m)-sP_{L^{m}}(m))\big]\; ,$$
$$v_{m,2}=\frac{m^{n+1}}{P_{E/L^{m}}(m)P(m)}\big[P(m)-2P_{E/L^{m}}(m)+\dfrac{2\delta(m)}{P(m)-s\delta(m)}((s-\epsilon_{L^{m}})
P(m)-sP_{E/L^{m}}(m))\big]\; ,$$ where $\epsilon(L^{m})$ is the number of times that the subsheaf $L^{m}$
appears on the minimal multi-index (c.f. (\ref{rightmuV})).

The set
$$
\mathcal{A}=\{\Theta_{m}:m\geq m_{4}\}\; ,$$
where
$$
\Theta_{m}(l)=(\mu_{v_{m}(l)}(\Gamma_{v_{m}(l)}))^{2}=||v_{m}(l)||^{2}\;
,
$$
is finite because the set of $m$-types, $\mathcal{P}$, is.  We say
that $f_{1}\prec f_{2}$ for two rational functions, if the
inequality $f_{1}(l)<f_{2}(l)$ holds for $l\gg 0$. Let $K$ be the
maximal function in the finite set $\mathcal{A}$, with respect to
the defined ordering. The function $K$ verifies that there exists
an integer $m_{5}$ such that, for all $m\geq m_{5}$, it is
$\Theta_{m}=K$ (c.f. \cite[Lemma 5.2]{GSZ}). 

\begin{prop}
\label{eventuallyrk2} Let $l_{1}$ and $l_{2}$ be integers with
$l_{1}\geq l_{2} \geq m_{5}$. Then, the $l_{1}$-Kempf filtration
of $E$ is equal to the $l_{2}$-Kempf filtration of $E$.
\end{prop}

\begin{pr}
By construction, the filtration
\begin{equation}
  \label{filt1}
  0\subset H^{0}(L^{l_{1}}(l_{1})) \subset H^{0}(E(l_{1}))
\end{equation}
is the $l_{1}$-Kempf filtration of $V\simeq H^{0}(E(l_{1}))$. Now
consider the filtration $V'_{\bullet}\subset V\simeq
H^{0}(E(l_{1}))$ defined as follows
\begin{equation}
\label{filt2}
0\subset H^{0}(L^{l_{2}}(l_{1})) \subset H^{0}(E(l_{1})) \; .
\end{equation}
We have to prove that \eqref{filt2} is, in fact, the $l_{1}$-Kempf
filtration of $V\simeq H^{0}(E(l_{1}))$.

Given that $l_{1},l_{2}\geq m_{5}$ we have
$\Theta_{l_{1}}=\Theta_{l_{2}}=K$. Hence,
$\Theta_{l_{1}}(l_{1})=\Theta_{l_{2}}(l_{1})$ and, by uniqueness
of the Kempf filtration (c.f. Theorem \ref{kempftheoremsheaves}),
filtrations \eqref{filt1} and \eqref{filt2} do coincide. Since,
in particular, $l_{1},l_{2}\geq m_{4}$, $L^{l_{1}}$ and $L^{l_{2}}$
 are $l_{1}$-regular by Proposition \ref{regularrk2}. Hence, $L^{l_{1}}(l_{1})$ and $L^{l_{2}}(l_{1})$ are
 generated by their global sections
 $H^{0}(L^{l_{1}}(l_{1}))$ and $H^{0}(L^{l_{2}}(l_{1}))$, respectively. By the
 previous argument, $H^{0}(L^{l_{1}}(l_{1}))=H^{0}(L^{l_{2}}(l_{1}))$, therefore $L^{l_{1}}(l_{1})=L^{l_{2}}(l_{1})$ and,
tensoring with $\mathcal{O}_{X}(-l_{1})$, this implies that
 the subsheaves $L^{l_{1}} \subset E$
  and $L^{l_{2}} \subset E$ coincide.
\end{pr}

Therefore, Theorem \ref{kempfstabilizesrk2} follows from Proposition \ref{eventuallyrk2}. Hence, eventually, the
Kempf filtration of the $\rk 2$ tensor $(E,\varphi)$ does not depend on the integer $m$.

\begin{dfn}
If $m\geq m_{5}$, the $m$-Kempf filtration of the $\rk 2$ tensor $(E,\varphi)$
$$0\subset (L,\varphi|_{L}) \subset (E,\varphi)$$ is called the \emph{Kempf filtration} or the \emph{Kempf subsheaf} of $(E,\varphi)$.
\end{dfn}

\section{Harder-Narasimhan filtration for $\rk 2$ tensors}
\label{HNdefinition}

Kempf theorem (c.f. Theorem \ref{kempftheoremsheaves}) says that, given an integer $m$ and $V\simeq
H^{0}(E(m))$, there exists a unique weighted filtration of vector spaces $V_{\bullet}\subseteq V$ which gives
maximum for the Kempf function
$$\mu(V_{\bullet},n_{\bullet})=\frac {\sum_{i=1}^{t+1} \frac{\Gamma_i}{\dim V} \big( r^i \dim V - r\dim V^i
+\frac{a_2}{a_1} (\epsilon^i(\overline\Phi)\dim V -s\dim V^i)\big)}
 {\sqrt{\sum_{i=1}^{t+1} {\dim V^{i}} \Gamma_{i}^{2}}}\; .$$ This filtration induces a unique rank $1$ subsheaf $L\subset E$ called
the Kempf subsheaf of the $\rk 2$ tensor $(E,\varphi)$. By
Proposition \ref{eventuallyrk2}, the subsheaf $L$ does not depend
on $m$, for $m\geq m_{5}$.

The Kempf function is a function on $m$ (c.f. Proposition \ref{identificationrk2}). Consider the
function
$$K(m)=m^{\frac{n}{2}+1}\cdot \mu(V_{\bullet},m_{\bullet})=\mu_{v_{m}}(\Gamma)\; .$$
Set $\gamma_{i}=\frac{r}{P}\Gamma_{i}$, then
$\frac{\gamma_{i+1}-\gamma_{i}}{r}=n_{i}$ and $\sum
r^{i}\gamma_{i}=\gamma_{1}+\gamma_{2}=0$, which gives
$\gamma_{1}=-n_{1}$, $\gamma_{2}=n_{1}$. Making the substitutions
for $m$ sufficiently large,
$$\dim V_{1}=\dim V^{1}=h^{0}(L(m))=P_{L}(m)\; ,$$
$$\dim V^{2}=\dim V-\dim V_{1}=h^{0}(E/L(m))=P_{E/L}(m)\; ,$$
we get
$$K(m)=m^{\frac{n}{2}+1}\cdot \frac {\sum_{i=1}^{2} \frac{\gamma_{i}}{r}[(r^i P - rP^i)+\frac{r\delta}{P-s\delta}
 (\epsilon^i P -sP^i)]}
 {\sqrt{\sum_{i=1}^{2} P^{i}\frac{P^{2}}{r^{2}} \gamma_{i}^{2}}}\; ,$$
where we set $P=P_{E}(m)$, $P^{1}=P_{L}(m)$, $P^{2}=P_{E/L}(m)$, $\epsilon^{1}=\epsilon(L)$,
$\epsilon^{2}=s-\epsilon(L)$. Note that $\epsilon^{i}=\epsilon^i(\overline\Phi)=\epsilon^{i}(\varphi)$ and
recall
$$\frac{a_{2}}{a_{1}}=\frac{r\delta}{P-s\delta}\; .$$
Substituting, we get
$$K(m)=m^{\frac{n}{2}+1}\cdot \frac{1}{P-s\delta}\frac {-n_{1}[2(\delta \epsilon^{1}-P^{1})+(P-\delta s)]
+n_{1}[2(\delta\epsilon^{2}-P^{2})+(P-\delta s)]}
 {\sqrt{P^{1}n_{1}^{2}+P^{2}n_{1}^{2}}}=$$
$$m^{\frac{n}{2}+1}\cdot \frac{r}{\sqrt{P}(P-s\delta)}[2P_{L}-P_{E}+\delta(s-2\epsilon(L))]\; .$$
Note that the unique weight $n_{1}$ does not appear in the function later from the substitutions, as it was
expected from a one-step filtration. Also note that the denominator of the function $K$ is positive (c.f.
choice of $m_{2}$ in proof of Proposition \ref{boundednessrk2}). Hence, we can state the following theorem.

\begin{thm}
\label{finalfunction} Given a $\delta$-unstable $\rk 2$ tensor $
(E,\varphi:\overbrace{E\otimes\cdots \otimes E}^{\text{s
times}}\rightarrow M) $, there exists a unique line
subsheaf $L \subset E$ which gives maximum for the polynomial
function
$$K(m)=2P_{L}(m)-P_{E}(m)+\delta(m)(s-2\epsilon(L))\; .$$
\end{thm}

If $X$ is a one dimensional complex projective variety, i.e. a smooth projective complex curve, we can simplify
the function $K$. Recall that, by Riemann-Roch, the Hilbert polynomial of a sheaf $E$ of rank $r$ and degree $d$
over a curve of genus $g$ is
$$P_{E}(m)=rm+d+r(1-g)\; ,$$
and the polynomial $\delta(m)$ becomes a positive constant that we
will denote by $\tau$. In this case, a coherent torsion free sheaf
of rank $2$ is a vector bundle of rank $2$ over $X$, and the Kempf
subsheaf will be a line subbundle.

\begin{thm}
\label{finalfunctioncurves} Given a $\tau$-unstable $\rk 2$ tensor
$ (E,\varphi:\overbrace{E\otimes\cdots \otimes E}^{\text{s
times}}\rightarrow M) $ over a smooth projective complex
curve, there exists a unique line subbundle $L \subset E$ which
maximizes the quantity
$$2\deg{L}-\deg{E}+\tau(s-2\epsilon(L))\; .$$
\end{thm}

Note that, if the tensor is unstable, such quantity will be positive, and the graph corresponding to the filtration will be \textit{a cusp} which is a convex graph.

If we define the \emph{corrected Hilbert polynomials} of $(E,\varphi)$ and $(L,\varphi|_{L})$ as
$$\overline{P}_{E}=P_{E}-\delta s\; ,$$
$$\overline{P}_{L}=P_{L}-\delta\epsilon(L)\; ,$$
we can rewrite the notion of stability for $\rk 2$ tensors (c.f. Definition \ref{stabilityrk2}): a $\rk 2$
tensor $(E,\varphi)$ is \emph{$\delta$-unstable} if there exists a line subsheaf $L\subset E$ such that
$$\frac{\overline{P}_{L}}{\rk L}>\frac{\overline{P}_{E}}{\rk E}\Leftrightarrow \overline{P}_{L}>
\frac{\overline{P}_{E}}{2}\; .$$

Theorem \ref{finalfunction} establishes that there exists a unique
subsheaf, the Kempf subsheaf, maximizing certain polynomial
function. This is equivalent to contradict, in a maximal way, the
definition of stability (c.f. Definition \ref{stabilityrk2}).
Therefore, we can define a notion of a Harder-Narasimhan
filtration for $\delta$-unstable $\rk 2$ tensors as this unique
line subsheaf which maximally contradicts GIT stability.

\begin{dfn}
\label{HNrk2tensors} If $(E,\varphi)$ is a $\delta$-unstable $\rk 2$ tensor, there exists a unique $\rk 1$ subsheaf
maximizing
$$2\cdot \overline{P}_{L}-\overline{P}_{E}>0\; .$$
We call
$$0\subset (L,\varphi|_{L})\subset (E,\varphi)$$ the \emph{Harder-Narasimhan filtration} of $(E,\varphi)$, and we
call $L$ the \emph{Harder-Narasimhan subsheaf} of $(E,\varphi)$.
\end{dfn}

\begin{rem}
\label{whynot}
We do not know, in principle, how to define a quotient tensor $(E/L,\overline{\varphi}|_{E/L})$, because we do
not know, a priori, how to define $\overline{\varphi}|_{E/L}$. This is why we cannot talk about quotient
tensors.

Given the exact sequence of sheaves, $0\rightarrow L\rightarrow E\rightarrow E/L\rightarrow 0$, we define the
corrected Hilbert polynomial of the quotient as $\overline{P}_{E/L}=\overline{P}_{E}-\overline{P}_{L}$, and we
have, trivially, the additivity of the corrected polynomials on exact sequences of sheaves. This way we can
consider that Definition \ref{HNrk2tensors} contains the analogous to the conditions of the classical
Harder-Narasimhan filtration for sheaves, in the case of $\rk 2$ tensors. Indeed,
$$2\cdot \overline{P}_{L}-\overline{P}_{E}>0\Leftrightarrow \overline{P}_{L}>\overline{P}_{E/L}\; ,$$
and the semistability of $(L,\varphi|_{L})$ and
$(E/L,\overline{\varphi}|_{E/L})$ (whichever definition of
$\overline{\varphi}|_{E/L}$ we impose), would follow trivially
from the fact that they are rank $1$ tensors.

Therefore, Definition \ref{HNrk2tensors} gives a notion of a Harder-Narasimhan filtration with the properties we
would expect it to have.
\end{rem}

\section{Stable coverings of a projective curve}
\label{coverings}
In this section we use the previous notions for
$\rk 2$ tensors over curves where the morphism is symmetric, and
the Definition \ref{HNrk2tensors} of the Harder-Narasimhan
subsheaf, to define stable coverings of a projective curve and,
for the unstable ones, a maximally destabilizing object, in terms
of intersection theory.

In the following, we shall consider tensors $(E,\varphi)$ where $E$ is a $\rk 2$ vector bundle over a
smooth complex projective curve $X$, and $$\varphi:\overbrace{E\otimes\cdots \otimes E}^{\text{s times}}\rightarrow M$$ 
is a symmetric non degenerate morphism. We call it a \emph{symmetric non degenerate rank $2$
tensor}. The non degeneracy condition means that $\varphi$ induces an injective morphism
$$E\hookrightarrow (\overbrace{E\otimes\cdots \otimes E}^{\text{(s-1) times}})^{\vee}\otimes M\; .$$
Let $\tau$ be a positive real number. Let $\mathbb{P}(E)$ be the projective space bundle of the vector
bundle $E$, which is a ruled algebraic surface (c.f. \cite[Section V.2]{Ha}).

The morphism $\varphi$ is, fiberwise, a symmetric multilinear map
$$\varphi_{x}:\overbrace{V\otimes\cdots \otimes V}^{\text{s times}}\rightarrow
\mathbb{C}\; ,$$ where $V\simeq \mathbb{C}^{2}$. Then, $\varphi_{x}$
factors through $\Sym^{s}(V)$, isomorphic to the
$(s+1)$-dimensional vector space of homogeneous polynomials of
degree $s$ in two variables. Hence, fiberwise, $\varphi$ can be
represented by a polynomial
\begin{equation}
\label{tensorfiber} \varphi_{x}\equiv
\sum_{i=0}^{s}a_{i}(x)X_{0}^{i}X_{1}^{s-i}
\end{equation}  which
vanishes on $s$ points in
$\mathbb{P}(V)\simeq\mathbb{P}_{\mathbb{C}}^{1}$. Therefore, as
$\varphi$ varies on $X$, it defines a degree $s$ covering
$$\mathbb{P}(E)\supset X'\rightarrow
X\; .$$

Suppose that $(E,\varphi)$ is a $\tau$-unstable $\rk 2$ tensor. Then, by
Theorem \ref{finalfunctioncurves}, there exists a line subbundle
$L\subset E$, the \emph{Harder-Narasimhan subbundle}, giving maximum for the quantity
\begin{equation}
\label{rk2stabexpression}
2\deg(L)-\deg (E)+\tau(s-2\epsilon(L))\; .
\end{equation}

The subbundle $L$ can be seen as a section of $\mathbb{P}(E)$,
each fiber $L_{x}$ corresponding to a point $P=\{L_{x}\}\in
\mathbb{P}_{\mathbb{C}}^{1}$. Recall from Definition
\ref{stabilitytensorsrk2} that $\epsilon (L)=k$ if
$\varphi|_{L^{\otimes (k+1)}\otimes E^{\otimes (s-k-1)}}= 0$ and
$\varphi|_{L^{\otimes k}\otimes E^{\otimes (s-k)}}\neq 0$. Note
that here we use the symmetry of the morphism $\varphi$.
Therefore, $\epsilon(L)=k$ means that, generically, $P=\{L_{x}\}$
is a zero of multiplicity $s-k$ and, by definition of the covering
$X'\rightarrow X$, $s-\epsilon(L)$ is, exactly, the number of
branches of $X'$ which generically do coincide with the section
defined by $L$, counted with multiplicity.

We can find in \cite{Gi} the classical example of classifying a
configuration of points in $\mathbb{P}_{\mathbb{C}}^{1}$
 up to the action of $\PGL(2)$. There, a homogeneous polynomial of degree $N$,
$P=\underset{i}\sum a_{i}X_{0}^{i}X_{1}^{N-i}$, is unstable if it
contains a linear factor of degree greater that $\frac{N}{2}$.
Now, observe that the restriction of a $\rk 2$ tensor to a point
$x\in X$ in (\ref{tensorfiber}), passing to the projectivization
$\mathbb{P}(E)$ hence fibers are isomorphic to
$\mathbb{P}_{\mathbb{C}}^{1}$, is precisely one of the homogeneous
polynomials in \cite{Gi}. Fiberwise, the morphism $\varphi$
defines a set of $s$ points in $\mathbb{P}_{\mathbb{C}}^{1}$. See
that, from the point of view of \cite{Gi}, letting $s=N$, the set
of points is unstable if there exists a point with multiplicity
greater that $\frac{s}{2}$.

Then, as $s-\epsilon(L)$ is the multiplicity of the point defined
by the line $L_{x}$ (the fiber of the Harder-Narasimhan subbundle
over $x$), in the set of $s$ points defined by the morphism
$\varphi$, following the previous argument, this point $\{L_{x}\}$
will destabilize the set if
$$s-\epsilon(L)>\frac{s}{2}\Leftrightarrow s-2\epsilon (L)>0\; ,$$
which is the second summand in (\ref{rk2stabexpression}). Hence,
the positivity of $s-2\epsilon(L)$ is equivalent to the line
bundle $L$ defining a point in the fiber
$\mathbb{P}_{\mathbb{C}}^{1}$, which coincides with one of the
zeroes of $\varphi$ in the fiber, and such that the zero has
multiplicity greater that $\frac{s}{2}$.

To conclude, we can say that the expression
(\ref{rk2stabexpression}) consists of two summands weighted by the
parameter $\tau$. First one, $2\deg(L)-\deg(E)$, is measuring the
stability of the vector bundle $E$. Second one, $s-2\epsilon(L)$,
is measuring the stability of the morphism or, with the previous
observations, the generic stability of the set of points defined
in $\mathbb{P}_{\mathbb{C}}^{1}$, fiberwise, as in \cite{Gi}, when
varying along the covering. Therefore, an object destabilizing a $\rk 2$ tensor is an object which contradicts these two
stabilities, weighted by $\tau$, and the Harder-Narasimhan
subbundle is the unique one which maximally does, for a
$\tau$-unstable tensor.

The sets of points in each fiber defined by $\varphi$ give a
covering of degree $s$,
$$\mathbb{P}(E)\supset X'\rightarrow X\; .$$
In the following, we rewrite the stability of the sets of points, fiberwise, as stability for the covering, using intersection theory for ruled surfaces.

\begin{prop}\cite[Proposition V.2.8]{Ha}
\label{normalized}
Given a ruled surface $\mathbb{P}(E)$, there
exists $E'\simeq E\otimes N$, with $N$ line bundle, such that
$H^{0}(E')\neq 0$ but for all line bundles $N'$ with negative
degree we have $H^{0}(E'\otimes N')=0$. Therefore,
$\mathbb{P}(E)=\mathbb{P}(E')$ and the integer $e=-\deg E'$ is an
invariant of the ruled surface. Furthermore, in this case, there
exists a section $\sigma_{0}:X\rightarrow \mathbb{P}(E')$ with
image $C_{0}$, such that $\mathcal{L}(C_{0})\simeq
\mathcal{O}_{X}(1)$.
\end{prop}

\begin{dfn}
Let $(E,\varphi:\overbrace{E\otimes\cdots \otimes E}^{\text{s times}}\rightarrow M)$ be a symmetric non
degenerate rank $2$ tensor over $X$. We call $(E',\varphi')$ an \emph{associated normalized tensor} where
$E'=E\otimes N$, $N$ a line bundle as in Proposition \ref{normalized}, and $\varphi'$ is the induced morphism given by
$$\varphi':(E')^{\otimes s}=E^{\otimes
s}\otimes N^{\otimes s}\rightarrow M\otimes N^{\otimes s}\; ,$$
and extending by the identity on $N^{\otimes s}$. 
\end{dfn}

\begin{prop}
\label{allnormalized}
The quantity in (\ref{rk2stabexpression}) is an invariant for all associated normalized tensors. Hence, $(E,\varphi)$ is 
$\tau$-unstable if and only if an associated normalized tensor $(E',\varphi')$ is $\tau$-unstable. 
\end{prop}
\begin{pr}
Let $N$
be a line bundle over $X$, as in Proposition \ref{normalized}. If we change $E$ by $E'=E\otimes N$,
then we have the line subbundle $L\otimes N\subset E'$ (by
exactness of the tensor product with locally free sheaves), and
$$\deg(E')=\deg(E\otimes N)=\deg(E)+2\deg(N)\; ,$$
$$\deg(L\otimes N)=\deg(L)+\deg(N)\; ,$$
so the quantity $2\deg(L)-\deg(E)$ is invariant by tensoring $E$
with a line bundle. 

Also note that, after defining  
$$\varphi':(E')^{\otimes s}=E^{\otimes
s}\otimes N^{\otimes s}\rightarrow M\otimes N^{\otimes s}\; ,$$ it
is $\epsilon'(L\otimes N)=\epsilon(L)$.

Hence, the quantity
$$2\deg(L)-\deg (E)+\tau(s-2\epsilon(L))$$
remains the same for associated normalized tensors. 
\end{pr}

Let $\mathbb{P}(E')$ be a ruled surface with $E'$ normalized as in Proposition \ref{normalized}. Let
$\sigma:X\rightarrow \mathbb{P}(E)$ be a section, and let $D=\im
\sigma$ be a divisor on $\mathbb{P}(E)$. It can be proved that
$\deg(\sigma)=-e-C_{0}\cdot D$, with these conventions (c.f.
\cite[Proposition V.2.9]{Ha}). Note that the section $C_{0}$ depends on the line bundle $N$ in Proposition \ref{normalized}, but 
the number $\deg(\sigma)=-e-C_{0}\cdot D$ does not. Let us define, by analogy,
$\epsilon(\sigma)=\epsilon(D)$ as the number of branches of $X'$
which generically do coincide with $D$, the divisor defined by
$\sigma$, counted with multiplicity.

\begin{dfn}
Let $f:X'\rightarrow X$ be a covering defined by a normalized symmetric non
degenerate rank $2$ tensor $(E,\varphi)$,
$X'\subset \mathbb{P}(E)$. Let $C_{0}$ be the image of a section $\sigma_{0}:X\rightarrow \mathbb{P}(E)$ such that $\mathcal{L}(C_{0})\simeq
\mathcal{O}_{X}(1)$. Let $\tau$ be a positive number. We say that $f$ is \emph{$\tau$-unstable} if there
exists a section $\sigma:X\rightarrow \mathbb{P}(E)$ with image $D$, i.e. there exists a line subbundle
$L\subset E$, such that the following holds
\begin{equation}
\label{stabintersection}
-2C_{0}\cdot D-e+\tau(s-2\epsilon(D))>0\; . 
\end{equation}
\end{dfn}

\begin{prop}
Let $\tau$ be a positive number. A symmetric non degenerate $\rk 2$ tensor $(E,\varphi)$ is $\tau$-unstable if and only if the associated
covering $f:X'\rightarrow X$ is $\tau$-unstable.
\end{prop}
\begin{pr}
By Proposition \ref{allnormalized} we can suppose that $(E,\varphi)$ is normalized and $\tau$-unstable. Then, we just have to notice
that expression (\ref{stabintersection}) corresponds to (\ref{rk2stabexpression}) by the previous
discussion, and does also not change by passing to a normalized associated tensor. 
\end{pr}

Finally, as we announced, we do characterize the Harder-Narasimhan filtration, in this case, in terms of
intersection theory. This last theorem follows from the previous results.

\begin{thm}
If $f:X'\rightarrow X$ is a degree $s$ covering coming from a symmetric non degenerate $\rk 2$ tensor
$(E,\varphi)$ which is $\tau$-unstable, then there exists a unique section
$\sigma:X\rightarrow \mathbb{P}(E)$ with image $D$, giving maximum
for
$$-2C_{0}\cdot D-e+\tau(s-2\epsilon(D))\, .$$
We call $\sigma$ the \emph{Harder-Narasimhan section} of the covering.
\end{thm}

\end{document}